\documentclass[12pt, reqno]{amsart}
\usepackage[active]{srcltx}
\usepackage{ifthen}
\theoremstyle{plain}

\usepackage[dvips]{graphicx}
\usepackage{color}
\definecolor{marin}{rgb}   {0.,   0.3,   0.7} 
\definecolor{rouge}{rgb}   {0.8,   0.,   0.} 
\definecolor{sepia}{rgb}   {0.8,   0.5,   0.} 
\usepackage[colorlinks,citecolor=marin,linkcolor=rouge,
            bookmarksopen,
            bookmarksnumbered
           ]{hyperref}

\newtheorem{thm}{Theorem}[section]
\newtheorem{cor}[thm]{Corollary}

\newtheorem{lem}[thm]{Lemma}
\newtheorem{prop}[thm]{Proposition}
\theoremstyle{definition} 
\newtheorem{defn}[thm]{Definition}

\numberwithin{equation}{subsection}

\newcommand{\D}[1]{\mbox{\rm #1}}

 \newcommand{\N}{\mathbb N}
 \newcommand{\R}{\mathbb R}

\newcommand{\cA}{\mathcal A}

  \newcommand{\cF}{\mathcal F}
  \newcommand{\cI}{\mathcal I}

 \newcommand{\cM}{\mathcal M}

 \newcommand{\cT}{\mathcal T}
 
\title[Convergence of solutions of the discrete discounted equation]
{Convergence of the solutions of the discounted equation: the discrete case}

\author[A. Davini, A. Fathi, R. Iturriaga, M. Zavidovique]
{Andrea Davini, Albert Fathi,\\ Renato Iturriaga, Maxime Zavidovique}
\address{Dip. di Matematica, {Sapienza} Universit\`a di Roma,
P.le Aldo Moro 2, 00185 Roma, Italy}
\email{davini@mat.uniroma1.it}
\address{UMPA, ENS-Lyon \& IUF, 46 all\'ee d'Italie, 69364 Lyon Cedex 7, France}
\email{albert.fathi@ens-lyon.fr}
\address{Cimat, Valenciana Guanajuato, M\'exico 36000}
\email{renato@cimat.mx}
\address{
IMJ-PRG (projet Analyse Alg\' ebrique), UPMC,  
4, place Jussieu, Case 247, 75252 Paris Cedex 5, France}
\email{zavidovique@math.jussieu.fr}
\numberwithin{equation}{section}
\thanks{Andrea Davini was partially supported by {CONACYT,} Mexico, research grant n.\,178838. \\
Albert Fathi, Renato Iturriaga and Maxime Zavidovique were supported by {ANR WKBHJ (ANR-12-BS01-0020)}}
\begin{document}
\maketitle
\begin{abstract}
We derive a discrete version of the results of \cite{Davini2016}. If $M$ is a compact metric space, $c : M\times M \to \R$ a continuous cost function and $\lambda \in (0,1)$, the unique solution to the discrete $\lambda$-discounted equation is the only function $u_\lambda : M\to \R$ such that
$$\forall x\in M, \quad u_\lambda(x) = \min_{y\in M} \lambda u_\lambda (y) + c(y,x).$$
We prove that there exists a unique constant $\alpha\in \R$ such that the family  of $u_\lambda+\alpha/(1-\lambda)$ is bounded as $\lambda \to 1$ and that for this $\alpha$, the family uniformly converges to a function $u_0 : M\to \R$ which then verifies
$$\forall x\in X, \quad u_0(x) = \min_{y\in X}u_0(y) + c(y,x)+\alpha.$$ 
The proofs make use of Discrete Weak KAM theory. We also characterize $u_0$ in terms of Peierls barrier and projected Mather measures.
\end{abstract}
\section{Introduction}

In \cite{Davini2016}, it was proven that the unique  viscosity solution of the $\lambda$-discounted Hamilton Jacobi equation converges, as $\lambda $ tends to zero,  to a particular solution of the critical Hamilton--Jacobi equation. In other words,  the limit selects one solution among the several possible choices.  In this work, we prove the discrete version of the same result. In this discrete setting, minimization of the action of curves is replaced  by minimization of  costs for  sequences, and the Hamilton--Jacobi equation by fixed points of the Lax--Oleinik semigroup.  This theory is  known as Discrete Aubry-Mather Theory. It was mainly developed  in \cite{BB} and \cite{Z}. 

Let $M$ be a compact metric space, and $c :M\times M \to \R$
a continuous function, that will be called cost function. 
The discrete version of  the  Hamilton--Jacobi equation $H(x,d_xu)=\alpha $ is  to find a $u\in C^0(M, \R)$ such that
\begin{equation}
\label{intro eqHJ}
u(x)= \cT (u)(x) +\alpha\qquad\hbox{for every $x\in M$},
\end{equation}
where $\cT$ is the Lax--Oleinik operator, defined on the set $C^0(M, \R)$ of continuous functions from $M$ to $\R$ as 
$$
\cT (g)(x) = \inf_{y\in M} g(y) + c(y,x) \qquad\hbox{for every $x\in M$ and $g\in C^0(M, \R)$}.
$$
Due to the compactness of $M$, there is only one constant $\alpha$ for which   we can find solutions of \eqref{intro eqHJ}.  This number  is called the critical value. As we will see, it has several characterizations. 

The discrete version of  the  $\lambda$-discounted Hamilton--Jacobi equation $\lambda u+ H(x,d_xu)=\alpha $ is  
\begin{equation}
\label{intro eqdisHJ}
u(x)= \cT_\lambda  (u)(x)+ \alpha\qquad\hbox{for every $x\in M$},
\end{equation}
where $\lambda$ is a parameter between 0 and 1, and $\cT_\lambda$ is an operator defined on $C^0(M, \R)$ as 
$$
\cT_\lambda  (g)(x) = \inf_{y\in M} \lambda g(y) + c(y,x) \qquad\hbox{for every $x\in M$ and $g\in C^0(M, \R)$}.
$$
Equation \eqref{intro eqdisHJ} admits a unique solution $u_\lambda$. Moreover, the family of solutions 
$(u_\lambda)_{0<\lambda<1}$ is equicontinuous and equibounded, see Proposition \ref{basicpropertiesdiscounted}. 
Clearly, any accumulation point of the $u_{\lambda}$, as $\lambda\to 1^-$, will be a solution of the  {discrete}   Hamilton--Jacobi equation.\footnote{In the continuous case we study the limit as $\lambda >0 $ tends to 0, and in the discrete case the limit  as $\lambda <1 $  tends to 1.}
Yet, equation  \eqref {intro eqHJ}  has several possible solutions, therefore it is not a priori clear whether the family $u_\lambda$ is fully convergent as $\lambda\to 1^-$.  The main Theorem of this work is to establish this convergence.

\begin{thm}
\label{maintheorem}
The solutions $u_\lambda$  of equation \eqref{intro eqdisHJ} converge, as $\lambda < 1$ tends to 1,  to a particular solution $u_0$ of equation  \eqref {intro eqHJ}.
\end{thm}

Stephane Gaubert pointed out that the result was already known when $M$ is finite. For example, it could be deduced from
\cite{Kohlberg}.

The definitions involved in the following paragraphs, subsolutions, projected Mather measures and Peierls barrier, will be given in the appendix.

Subsolutions of \eqref {intro eqHJ} do not need to be continuous, however, according to Proposition \ref{integrability}, all subsolutions of \eqref {intro eqHJ} are integrable with respect to any projected Mather measure  {(see Definition \ref{defMather})}.   We denote by $\cF_- $ the set of subsolutions  such that
$$ \int_M u(x)\,d\mu(x)  \le 0$$
for all projected Mather measures. We have the following characterizations for the solution selected in the limit.

\begin{prop}
\label{main.prop} 
The limit solution  $u_0$  in Theorem \ref{maintheorem} can be  characterized  in either of the following  two  ways:\\ 
\begin{enumerate}
\item $\qquad\qquad u_0(x)= \sup\limits_{u\epsilon \cF_-} \, u(x),$\medskip\\
\item 
$\qquad\qquad\displaystyle u_0(x) =\min\limits_{\mu}\int_M h(y,x)\,d\mu(y),$\medskip\\
where $h$ is the Peierls barrier and $\mu $ varies in the set of projected Mather measures.
\end{enumerate}
  \end{prop}

 For convenience of the reader, we will also state, in the appendix,   the results we  use  from  Discrete Aubry--Mather Theory, see  \cite{BB} or \cite{Z} where proofs can be found. The non-expert reader should probably first look at the appendix.\\  
 
\smallskip\indent{\textsc{Acknowledgement. $-$}}
Part of this work was done while the first (AD) and second (AF) authors were visiting CIMAT  in Guanajuato, that they both wish to thank for its hospitality.  The final version was done while the second author was visiting DPMMS, University of Cambridge. The authors thank the anonymous referee for useful comments that helped improve the presentation of the present paper.

\section{Preliminaries}
  
In this section we shall state and prove some preliminary facts about the discounted equation 
\begin{equation}\label{beta eqdisHJ}
u(x)= \cT_\lambda  (u)(x)+ \beta=\inf_{z\in M}\lambda u(z)+c(z,x)+\beta,
\end{equation}
for every $x\in M$, where $\beta$ is a fixed real constant. 

Let $u$ be a continuous function on $M$. We will say that $u$ is a {\em subsolution} of \eqref{beta eqdisHJ} if \ $u(x)\leq \cT_\lambda(u)(x)+\beta$ for every $x\in M$. We will say that $u$ is a {\em supersolution} of \eqref{beta eqdisHJ} if \ $u(x)\geq \cT_\lambda(u)(x)+\beta$ for every $x\in M$.  

In the sequel, we shall denote by  $S_{n} (x)$ the set of $M$-valued sequences of the form $(x_{-n},x_{-n+1},\dots,x_{-1},x_0)$ with $x_0=x$, and by 
$S_{\infty} (x)$ the set of $M$-valued sequences $\bar x =(x_{-n})_{n\geq 0}$ such that $x_0= x$. 
It will also be convenient to set $S_{n} (M)=\cup_{x\in M}S_{n} (x)$, and $S_{\infty} (M)=\cup_{x\in M}S_{\infty} (x)$.

One of our main tools is the following comparison principle:

\begin{prop}\label{prop comparison}
Let $v$ and $w$ be a pair of continuous functions on $M$ that are, respectively, a sub and a supersolution of \eqref{beta eqdisHJ}. Then 
\begin{eqnarray}\label{claim comparison}
v(x)\leq \min_{\bar x \in S_\infty(x)} \sum_{n=0}^{\infty}\lambda ^{n} \big(c(x_{-n-1},x_{-n} )+\beta\big) \leq w(x)
\end{eqnarray}
for every $x\in M$.
\end{prop}

\begin{proof}
Replacing the cost function $c(x,y)$ with $c(x,y)+\beta$, we can always assume that $\beta=0$. 
Let us pick a point $x\in M$. By the definition of $\cT_\lambda$ and the fact that $v\leq \cT_\lambda(v)$ on $M$ we get 
\begin{eqnarray*}
v(x)&\leq& \min_{x_{-1}}\big(\lambda v(x_{-1})+c(x_{-1},x)\big) \leq \min_{x_{-1}}\big(\lambda \cT_\lambda(v)(x_{-1})+c(x_{-1},x)\big)\\
&=&\min_{x_{-2},x_{-1}}\big(\lambda^2 v(x_{-2})+\lambda c(x_{-2},x_{-1})+c(x_{-1},x)\big).
\end{eqnarray*}
Arguing inductively, we derive 
\[
v(x)\leq \min_{\bar x\in S_\infty(x)}\Big(\lambda^{n}v(x_{-n})+\sum_{k=0}^{n-1} \lambda^k c(x_{-k-1},x_{-k}) \Big).
\]
Since $v,c$ are continuous functions defined on compact spaces, and $\lambda<1$, the sequence of continuous 
functions $\lambda^{n}v(x_{-n})+\sum\limits_{k=0}^{n-1} \lambda^k c(x_{-k-1},x_{-k})$ converges uniformly to $\sum\limits_{k=0}^{\infty} \lambda^k c(x_{-k-1},x_{-k})$ on the compact space $S_\infty(x)$. Therefore, the left hand side 
inequality in \eqref{claim comparison} holds. The inequality for $w$ follows arguing analogously.
\end{proof}

The existence of a (unique) solution of equation \eqref{beta eqdisHJ} is established in the next proposition. 

\begin{prop}
For $0< \lambda <1 $
there is only one solution $u_{\lambda}$ of   the discounted equation \eqref{beta eqdisHJ}  and it can be represented by  
\begin{equation}
\label{solutionulambda}
u_\lambda (x_0)= \min_{\bar x \in S_\infty(x_0)} \sum_{n=0}^{\infty}\lambda ^{n} \big(c(x_{-n-1},x_{-n} )+\beta\big)\quad\hbox{for every $x_0\in M$}. 
\end{equation}

 \end{prop}
 
 \begin{proof}
As before,  to simplify notations, replacing the cost function $c(x,y)$ with $c(x,y)+\beta$, we will assume $\beta=0$. 
 For $\lambda $ strictly smaller than 1, the operator $u\mapsto \cT_\lambda $ is a contraction in the space of continuous functions with the $C^0$--norm. Indeed, let $f$ and $g$ be two continuous functions. For a given $x$ in $M$, let $y$ such that $\cT_\lambda f(x) = \lambda f(y) + c(y,x)$. By definition we have $\cT_\lambda g(x) \le \lambda g(y) + c(y,x)$, so 
 
 $$\cT_\lambda g(x)- \cT_\lambda f(x)
  \le \lambda (g(y)-f(y))  \le \lambda \|f-g\|_0,$$
 where $ \| \cdot  \|_0 $ denotes  the $C^0$ norm.
  Reversing the roles of $f$ and $g$  we obtain
  $$|\cT_\lambda g(x)- \cT_\lambda f(x)|
    \le \lambda \|f-g\|_0.$$
  Since this is true for every $x$, we obtain 
  $$\|  \cT_\lambda f- \cT_\lambda g \|_0 \le \lambda \|f-g\|_0 .$$
 Therefore,  from the Banach fixed point theorem, we obtain  that there is a unique fixed point $u_\lambda$. The representation formula \eqref{solutionulambda}   is a direct consequence of Proposition \ref{prop comparison}. Alternatively, since iterates of the contraction map converge to the fixed point,   we can obtain the same formula   as the limit, as $n$ tends to  infinity, of $\cT_\lambda^n (0)$.
\end{proof}

Some crucial properties of the solutions of the discounted equation are established in the next proposition. It incidentally entails a characterization for the critical value $\alpha$  {(see Theorem \ref{weakdiscretekam})}.  

\begin{prop}\label{basicpropertiesdiscounted}
For every $\beta\in\R$, the family $\{u^\beta_\lambda\,:\, 0<\lambda <1\,\}$ of solutions of \eqref{beta eqdisHJ}  is equicontinuous. Furthermore, it is equibounded if and only if $\beta$ is equal to the critical value $\alpha$. 
\end{prop}

\begin{proof} 
For the first part, let $x$   and $y$ be two points in $M$ and $z$ a point realizing the infimum 
of the discounted Hamilton--Jacobi equation \eqref{beta eqdisHJ} for the point $x$.  We obtain
 $$u^\beta_\lambda(y)-u^\beta_\lambda(x) \le c(z,y)- c(z,x),$$
so the solutions $u^\beta_{\lambda}$ have all the same modulus of continuity as the cost function $c$.

Let us prove that they are equibounded when $\beta$ equals the critical constant $\alpha$. Take a solution $u$ of 
\begin{equation}\label{eq critica}
u= \cT (u) +\alpha.
\end{equation}
Since $u$ is continuous and $M$ is compact, we can find a constant $k$ such that ${\overline u}(x):=u(x)+k$ and ${\underline u}(x):=u(x)-k$ are a positive and negative solution of \eqref{eq critica}. It is then easily seen that
\begin{align*}
\underline u(x)&=\cT (\underline u)(x)+\alpha\leq \cT_\lambda(\underline u)(x)+\alpha,\\
\overline u (x)&=\cT (\overline u)(x)+\alpha\geq \cT_\lambda(\overline u)(x)+\alpha,
\end{align*}
for $x\in M$, and $0<\lambda<1$. Namely, for every $0<\lambda<1$, the continuous functions $\underline u$ and $\overline u$ are, respectively, a subsolution and a supersolution of \eqref{beta eqdisHJ}, with $\beta=\alpha$. By the comparison principle stated in Proposition \ref{prop comparison}, we conclude that $\underline u\leq u^\alpha_\lambda \leq \overline u$ on $M$ for every $0<\lambda<1$. This implies  that the family $(u^\alpha_\lambda)_{ 0<\lambda<1}$ is equibounded.

To prove the {\em only if} part, it is enough to observe that $u^\beta_\lambda=u^\alpha_\lambda -\frac{\alpha -\beta}{1-\lambda}$.
\end{proof}

\section{Proof of the main Theorem}

In this section, we will prove  both Theorem \ref{maintheorem} and the first characterization given in Proposition \ref{main.prop}. 

Again, to simplify notations, we will assume in the sequel that the critical value $\alpha$ is equal to 0. As previously noted, this does not affect the generality. Therefore the discounted equation rereads as 

\begin{equation}
\label{eqdisHJ}
u= \cT_\lambda  (u),
\end{equation}
where $\lambda$ is a real parameter such that $0<\lambda<1$. We will denote by $u_\lambda$ the unique solution of equation \eqref{eqdisHJ}.  Since we are assuming $\alpha=0$, the discrete version of the critical Hamilton--Jacobi equation is 
\begin{equation}
\label{eqHJ}
u= \cT (u).
\end{equation}
Let us denote by $\cM_0$ the set of projected Mather measures (on $M$) for \eqref{eqHJ} (see the appendix,   Definition \ref{defMather}) and set 
$u_0(x):= \sup_{u\in \cF_-}  u(x)$ for every $x\in M$, where $\cF_{-}$ is the set of subsolutions $u$ of \eqref{eqHJ} such that 
\[
\int_{M} u(x)\,d\mu(x)\leq 0\quad\hbox{for every $\mu\in\cM_{0}$}.
\]
The following holds:

\begin{prop}
\label{prop.promedio} 
Let  $u$ be a  limit point of the functions $u_\lambda $,  as $\lambda \to 1^-$. Then,   $u\in \cF_{-}$, i.e.    for any measure $\mu $ in $\cM_0$ we have 
$$\int_M u(x)\, d\mu(x) \le 0.$$
In particular, $u \le u_0$.
\end{prop}

\begin{proof}
 Since  $u_\lambda $ is a solution of  (\ref{eqdisHJ})  we have   
  $$u_\lambda(x)-\lambda u_\lambda(y) \le c(y,x) \qquad\hbox{for every $(x,y)\in M\times M$}. $$
 Integrating the inequality above with respect to a Mather measure $ \tilde\mu $ defined on $M\times M$ yields
 $$\int_{M\times M} \big(u_\lambda(x)-\lambda u_\lambda(y)\big)\,d \tilde\mu(x,y) \le \int_{M\times M} c(y,x)\,d \tilde\mu(x,y). $$
 Since $\tilde\mu $ is a Mather measure  the right hand side of this  inequality is zero. Therefore, we have 
 $$(1-\lambda )\int_M   u_\lambda(x)\,d \mu(x) \le 0 , $$
 where $\mu$ is the projection of $\tilde\mu$ on either the first or second factors of $M\times M$.  Dividing by $1-\lambda >0$  we conclude that $\int_M   u_\lambda(x)d \mu(x) \le 0  $. So if $u$ is a uniform  limit of $u_{\lambda_i}$ for some sequence $\lambda_i\to 1^{-}$, we obtain the first assertion. Since $u$ is a solution, it is in particular a subsolution; therefore  $u \le u_0$. 
\end{proof}

 To prove the other inequality, we need to introduce a special  class of measures. 
 Given a sequence  $\bar x $ in $S_{\infty} (M)$  and a positive  $\lambda <1 $, we denote by 
 $\tilde \mu_{\bar x}^\lambda$ the probability measure in $M \times M$ defined as 
 $$\int_{M\times M}f(x,y) \,d\tilde \mu_{\bar x}^\lambda(x,y)=  a_\lambda \sum_{n=0}^{\infty}\lambda ^n f(x_{-n-1},x_{-n} ) $$ 
 for any continuous function $f:M\times M\to \R$, where $a_\lambda= 1-\lambda $.  The choice of the constant $ a_\lambda $ guarantees that $ \tilde \mu_{\bar x}^\lambda $ is a probability measure.

  \begin{lem}
\label{constr.mather}
Let $ (\bar x^\lambda ) $,   with $0< \lambda <1 $,  be a family  of sequences in $S_\infty(M)$. 
If $\tilde\mu $ is an accumulation point of $ \tilde \mu_{\bar x^\lambda}^\lambda$ as $\lambda$  tends to $1$, then $\tilde\mu $  is a closed measure. Moreover, if  for every $\lambda$ the sequence $\bar x^\lambda $  realizes the infimum in \eqref{solutionulambda},
then the  measure  $\tilde\mu $   is a Mather measure. 
\end{lem}

\begin{proof}

For the first part it is enough to show that for any continuous function $\phi :M\to \R $ we have
$$\int_{M\times M} \big(\phi(x)-\phi(y)\big) \,d\tilde \mu(x,y) =0.$$

We have that 
\begin{align}
\nonumber   \int_{M\times M} \big( \phi(y)-\phi(x) \big)\,d\tilde \mu_{\bar x^\lambda}^\lambda(x,y) &=
a_\lambda \sum_{n=0}^{\infty}\lambda ^{n} \big(\phi(x^\lambda_{-n})-\phi (x^\lambda_{-n-1} )\big) \\
\nonumber   &=a_\lambda \Big[ \sum_{n=0}^\infty\lambda ^{n} \phi(x^\lambda_{-n}) -\sum_{n=1}^{\infty} \lambda ^{n-1} \phi(x^\lambda_{-n})\Big]\\
\nonumber   &= a_\lambda \Big[ \phi(x^\lambda_{0})+(\lambda-1)\sum_{n=0}^\infty\lambda ^{n} \phi(x^\lambda_{-n-1}) \Big] \\
\label{exprsimple}&= a_\lambda\Big( \phi(x^\lambda_{0}) - \int_{M\times M}\phi(x)\,  d\tilde\mu_{\bar x^\lambda}^\lambda(x,y)\Big).
\end{align}
It follows that $\Big|\int_{M\times M} \big( \phi(x)-\phi(y) \big)d\tilde \mu_{\bar x^\lambda}^\lambda  \Big|   \le   2a_\lambda   \|\phi\|_0$,
where $ \|\phi\|_0 $ is, as before, the $C^0$ norm of $\phi$. Since $a_\lambda=(1-\lambda)\to 0$ as $\lambda \to1$, we obtain the first part of the lemma.

Furthermore, if for every $0< \lambda < 1$ the sequence $\bar x^\lambda $ realizes the infimum in \eqref{solutionulambda}, we have
\begin{equation}\label{minimizingseq}  \int_{M\times M}  c(x,y) \,d\tilde \mu_{\bar x^\lambda}^\lambda(x,y)  =     a_\lambda \sum_{n=0} ^{\infty} \lambda ^{n} c(x^\lambda_{-n-1}, x^\lambda_{-n} )
     =   a_\lambda u_\lambda (x^\lambda_0).
\end{equation}
 Since the functions $u_\lambda$ are equibounded, and $a_\lambda\to 0$ as $\lambda\to 1$,  for any limit measure $\tilde \mu $ of  $\mu_{\bar x^\lambda}^\lambda $, we obtain 
$\int_{M\times M}  c(x,y)\, d\tilde \mu(x,y)=0$.
Hence $\tilde\mu$ is a Mather measure. 
\end{proof}

The following is a key lemma for the end of the proof and for the second characterization.
  
\begin{lem}
\label{KeyLemma}
Suppose $w$ is a  continuous subsolution of \eqref{eqHJ}. If $\bar x^\lambda=(x^\lambda_{-n})_{n\geq 0}$ is a minimizing sequence for \eqref{solutionulambda}, we have
$$u_\lambda (x^\lambda_0) \ge  \Big(w(x^\lambda_0) - \int_{M\times  M} w(x)\, d\tilde \mu_{\bar x^\lambda}^\lambda(x,y)\Big).$$
\end{lem}

\begin{proof}
By using the fact that $\bar x^\lambda$ realizes the minimum in \eqref{solutionulambda},  together with \eqref{minimizingseq} and \eqref{exprsimple}, we get
\begin{align*}
u_\lambda (x^\lambda_0) & =  \frac{1}{a_\lambda} \int_{M\times M}  c(x,y) \,d\tilde \mu_{\bar x^\lambda}^\lambda(x,y)\\
&  \ge  \frac{1}{a_\lambda}   \int_{M\times M} \big( w(y)- w(x)\big)\, d\tilde \mu_{\bar x^\lambda}^\lambda(x,y) \\
& =  w(x^\lambda_0) -  \int_{M\times M} w(x)\,d \tilde \mu_{\bar x^\lambda}^\lambda(x,y). 
\end{align*}
\end{proof}

 In the above lemma, the continuity of the subsolution is used to  integrate it with respect to the measure $\mu_{\bar x^\lambda}^\lambda$.\smallskip \par

We derive from Lemma \ref{KeyLemma} the following asymptotic result:

\begin{prop}
\label{keycor}
Suppose $u$ is a uniform limit of a  subsequence of  $u_\lambda $  as $\lambda $ converges to 1.
Then, for every (possibly discontinuous) subsolution $w$, we have %
\begin{equation}\label{ineqfondamentale}
u \ge  w-\sup_{\mu\in\cM_{0}} \int_{M} w(x)\, d \mu(x) .
\end{equation}
Therefore $u\geq u_0$.%
\end{prop}
\begin{proof} We first prove \eqref{ineqfondamentale} when $w$ is continuous. We fix $x_0\in M$, and find, for $0<\lambda<1$,
a sequence $\bar x^\lambda=(x^\lambda_{-n})_{n\geq 0}$ minimizing  \eqref{solutionulambda}, with $x^\lambda_0=x_0$. For each such sequence
$\bar x^\lambda$, we consider the probability measure
$\tilde \mu_{\bar x^\lambda}^\lambda$ on $M\times M$. Extracting further, we can assume that $\tilde \mu_{\bar x^{\lambda_i}}^{\lambda_i}$ converges weakly to a measure $\tilde \mu$ on $M\times M$. From Lemma \ref{constr.mather},
we know that $\tilde \mu$ is a Mather measure. By Lemma \ref{KeyLemma}, we have
$$u_{\lambda_i} (x_0) \ge  w(x_0) -  \int_{M\times  M} w(x)\,
d\tilde \mu_{\bar x^{\lambda_i}}^{\lambda_i}(x,y).$$
As $i\to +\infty$,  the sequence $\lambda_i$ converges to $ 1$, the functions $u_{\lambda_i}$ converge uniformly to $u$, and the probability measures $\tilde \mu_{\bar x^{\lambda_i}}^{\lambda_i}$ converge  to  $\tilde \mu$. Therefore we can pass to the limit in the inequality above, we obtain
$$u (x_0) \ge  w(x_0) -  \int_{M} w(x)\,d\mu(x),$$
where $\mu=\pi_1^*\tilde  \mu $ is a projected Mather measure. This proves \eqref{ineqfondamentale} when $w$ is continuous.

Let us consider the case of a not necessarily continuous subsolution $w$. By Proposition \ref{continuitysol}, we know $\cT(w)$ is a continuous subsolution of \eqref{appendix eqHJ} such that $\cT(w)\geq w$,  and $\cT(w)=w$ on $\cA$. Therefore, by Proposition 
\ref{supportcontained}, we get $\int_{M} w\,d\mu=\int_{M} \cT(w)\,d\mu$,
for every projected Mather measure $\mu$.
Since \eqref{ineqfondamentale} is true for the continuous subsolution $\cT(w)$, we obtain 
\begin{align*}
u &\ge  \cT(w) -  \int_{M} \cT(w)(x)\,d\mu(x),\\
&\ge  w -  \int_{M} w(x)\,d\mu(x).
\end{align*}

If $w\in \cF_-$, we note that $\int_{M} w(x)\,d\mu(x)\leq 0$, for every projected Mather measure $\mu$. Therefore,
in this case, the inequality  \eqref{ineqfondamentale} implies $u\geq w$. Taking the sup over all $w\in \cF_-$ yields
$u\geq u_0=\sup_{w\in \cF_-} w$.
\end{proof}

We now finish  the proof of Theorem  \ref{maintheorem}. 

\begin{proof}   [Proof of Theorem  \ref{maintheorem} ]
The family $(u_\lambda)_{0<\lambda<1}$ is equicontinuous, hence relatively compact in the C$^0$ topology by 
the Arzel\`a-Ascoli theorem. It therefore suffices to show that any 
limit of a  subsequence of solutions $u_\lambda$, as $\lambda $ converges to 1, is necessarily equal to $u_0$.
Let $u$ be such a limit. From Proposition \ref{prop.promedio},  we obtain $u\leq u_0$,  and from Proposition \ref{keycor}
we obtain $u_0\geq u$. 

\end{proof}

\section{Proof of  the second Characterization}

{  In this section, we will prove  the second characterization stated in Proposition \ref{main.prop}, namely that the function $u_0$ obtained as the limit of the solutions of the discounted equations coincide with the function $\hat u_0 :M\to \R $ defined by
$$\hat u_0(x)=  \min_{\mu\epsilon \cM_0}\int h(z,x)\,d\mu(z)\qquad\hbox{for every $x\in M$}.$$

We start by proving the following results:}
\begin{lem}
\label{lemma bof}
The function $\hat u_0$  is a subsolution.
\end{lem}

\begin{proof}
For $\mu\in\cM_0$,  define the function 
$$h_\mu (x) =\int_ M h(y,x)\, d\mu (y)\qquad\hbox{for every $x\in M$}. $$
Since $h_\mu$ is the convex combination of {  an equibounded family of} solutions of \eqref{eqHJ}, see (1) of Proposition \ref{solpeierls}, by (1) of Proposition \ref{basicprop}, it is itself a subsolution of \eqref{eqHJ}. As an infimum of {\sl  equibounded} subsolutions, we infer, by the second item of the same proposition,   that $\hat u_0$ is also a subsolution of \eqref{eqHJ}.
\end{proof}

 \begin{lem}
\label{lema<}
We have that  $u_0 \le \hat u_0$  

\end{lem}

\begin{proof}
By  definition  of  $u_0 $ and $ \hat u_0$,  we only need to show that $u \le h_\mu$ on $M$, 
where $u$ is a subsolution in $\cF_-$ and $\mu$ is a projected Mather measure.
By  proposition \ref{subpeierls} we have 
 $$u(x)-u(z) \le h(z,x) \qquad\hbox{for every $(x,z)\in M\times M$}. $$
Integrating with respect to $\mu\in\cM_0$ we obtain
$$u(x) - \int_M u(z) \,d\mu (z) \le h_\mu(x) . $$ 
Using that $ \int_M u(z) d\mu (z) $ is non-positive, we get the assertion.
\end{proof}

{  We are now ready to prove the announced equality.}

\begin{proof}   [Proof of Proposition  \ref{main.prop} ]
By the previous lemma we already know that  $u_0 \le \hat u_0$. We have to show the reverse inequality $u_0 \ge \hat u_0$. Since we know that $u_0$ is a solution and that $\hat u_0$ is a subsolution, by Proposition 
 \ref{Maxprinciple} we only have to show the inequality on the projected  Aubry set $\cA$.
 
 Fix $y$ in  $\cA$.  By Proposition \ref{solpeierls},  the function $x\mapsto -h(x, y) $ is a  subsolution of \eqref{eqHJ}. Hence, adding the constant $\hat u_0 (y) $ to this function, we obtain that
 $$w(x) = -h(x,y) +\hat u_0 (y), $$
is  a subsolution, which is clearly  in $\cF_-$. So $u_0 \ge w$, in particular, by evaluating at $y$, we get
 $$u_0(y) \ge -h(y,y ) + \hat u_0(y).$$ 
Since $h(y,y) =0$ on the projected Aubry set $\cA$, this yields $u_0\ge \hat u_0$ on $\cA$ and therefore also on $M$.
\end{proof}

   \begin{appendix}
   \section{Discrete Aubry-Mather Theory}
\label{sec:discrete}

\subsection{Lax Oleinik Operator}

Let  $M$ be a compact metric space, and 
$$c:M\times M \to \R $$ be a continuous cost  function. For every function $u:M\to\R$ (not necessarily continuous), we set 
\begin{equation}\label{appendix T}
\cT (u)(x) = \inf_{y\in M} u(y) + c(y,x) \qquad\hbox{for every $x\in M$}.
\end{equation}
When computed for functions $u$ that are continuous on $M$, the operator $\cT$ is the Lax--Oleinik operator as defined in the introduction. As a direct consequence of the definition, we derive

\begin{prop}\label{app prop cont}
Let $u$ be a (possibly discontinuous) real function on $M$ such that $\cT(u)(x)>-\infty$ for every $x\in M$. Then $\cT(u)$ is continuous on $M$, with the same continuity modulus as $c$. 
\end{prop}

Since $c$ is continuous, and therefore bounded on the compact space $M\times M$, it is not difficult to see that $\cT(u)$
is bounded below if and only if $u$ is bounded below, and also if and only if $\cT(u)(x)>-\infty$ for some $x\in M$.

The following Theorem  is the discrete version of the Weak KAM theorem a proof of which 
can be found in \cite[Theorem 1.2]{Z}

\begin{thm}
\label{weakdiscretekam}
There is a  unique constant $\alpha$ such that the equation
$$ 
  u= \cT(u)+\alpha
$$
  admits (necessarily continuous) solutions $u:M\to\R$. 
\end{thm}

Such a constant $\alpha$ is called the {\em critical value}.

\subsection{Mather measures}  For $i =1,2$, let  $\pi_i$ be the projection  to each factor of $M\times M$, that is  $( x_1,x_2) \mapsto x_i$.  For a probability measure  $\tilde \mu $  in $M\times M$,  the projected measures $\pi_i^* \tilde \mu$ are given by the formula
$$\int_M f(x)\, d\pi_i^*\tilde\mu(x)= \int_{M\times M} (f\circ \pi_i)(x,y)\, d \tilde \mu(x,y).$$
A Borel probability measure  in $M\times M$ is called  {\it closed } if 
the projections are the same, i.e. $\pi_1^*\tilde  \mu = \pi_2^*\tilde  \mu $. A proof of the following proposition can be found in \cite[Theorem 15]{BB}.
\begin{prop}\label{app prop Mather measure}
We have
\begin{equation}\label{app eq Mather measure}
-\alpha =\min_{\nu}\int_{M\times M} c(x,y)\,d\tilde\nu(x,y),
\end{equation}
where the minimum is computed for $\tilde\nu$ in the class of closed probability measures on $M\times M$.
\end{prop}

 {\begin{defn}\label{defMather} \rm A closed probability measure that is a solution of the minimization problem \eqref{app eq Mather measure} is called 
 { \it Mather measure}. The set of Mather measures  on $M\times M $ is denoted by $\cM$ and the set of projected Mather measures is denoted by $\cM_0$.
 \end{defn}
 }

   Measures on $M\times M $  will be denoted with a tilde, while for the projection we will use the same letter without the tilde.

\subsection{Peierls Barrier and Aubry set}
For each pair of points $x$ and $y$ let $S_n(x,y)$ be the set of  $M$-valued  sequences of the form $\bar x=(x_{-n},x_{-n+1},\dots,x_{-1}, x_{0})$ such that $x_{-n}=x$ and $x_{0}=y$. Denote by $c(\bar x)$ the cost of the sequence, that is,   $ \sum\limits_{i=1}^{n}c(x_{-i},x_{-i+1})$.
  For each natural number $n$ define the function 
$$c_n(x,y) =\inf_{\bar x \in  S_n(x,y)} c(\bar x). $$
For each  real value $\kappa$, we define the functions
$$h_n^\kappa(x,y) =c_n(x,y) +n\kappa,$$
$$h^\kappa(x,y) =\liminf_{n \to \infty}  h^\kappa_n(x,y).$$
The function $h^\kappa$ is finite valued if and only if $\kappa$ equals the critical value $\alpha$, see \cite{Z}. 
The function $h^\alpha $ is called the Peierls barrier. As noted before, up to adding a constant to the cost, we can always suppose that $\alpha$ is zero. To simplify notations, this will be always assumed in the rest of this appendix. Moreover, we will write $h$, $h_{n}$ instead of $h^\alpha$, $h^{\alpha}_n$, respectively.

Since 
\begin{equation}\label{triangle}
h_{n+m}(x,z) \le h_n(x,y) + h_m(y,z) \quad \hbox{for every positive integers $m$ and $n$,}
\end{equation}
 it is easy to show that $h$ satisfies the triangle inequality. The symmetrized function $h(x,y)+h(y,x)$  is nonnegative, but,
 usually, it is not a distance function because in general $h(x,x) $ can be positive and
there might be different points such that $ h(x,y)+h(y,x) =0$. Partially motivated by this, we define the projected Aubry $\cA$  by
$$
\cA:=\{y\in M\,:\,h(y,y) =0\,\}.
$$

\subsection{Subsolutions and supersolutions}
We consider the following discrete version of the Hamilton--Jacobi equation:
\begin{equation}\label{appendix eqHJ}
u(x)=\cT(u)(x)\quad\hbox{for every $x\in M$.}
\end{equation}
A  function  $u$ defined on $M$ is called subsolution of \eqref{appendix eqHJ} if 
$$ u \le  \cT (u). $$
A  function  $u$ defined on $M$ is called supersolution of \eqref{appendix eqHJ} if
$$ u \ge  \cT (u). $$
A   function  $u$ defined on $M$ is called solution  if it is both a subsolution and supersolution, i.e. if it is a fixed point of the operator $\cT$.
The following holds:
\begin{prop}
\label{continuitysol}
Let $u$  be a subsolution of \eqref{appendix eqHJ}, which is bounded from below. Then $\cT(u)$ is a continuous subsolution of \eqref{appendix eqHJ} such that $\cT(u)\geq u$ on $M$, and $\cT(u)=u$ on $\cA$. 
In particular, $u$ is continuous on the projected Aubry set $\cA$.
\end{prop}

\begin{proof}
The first assertion has been proven in \cite[Proposition A.10]{Z}. This immediately gives the continuity of $u$ on $\cA$ in view of Proposition \ref{app prop cont}.
\end{proof}

      The following is mostly contained in  \cite[Lemma 2.32 and Proposition A.8.3]{Z}.
    
\begin{prop}\label{basicprop}\ 
\begin{itemize}
\item[1)]  A convex combination of {  an equibounded family of} subsolutions is a subsolution.\medskip
\item[2)]   If $\{u_i\}_{i\in \cI}$  is an {  equibounded} family of subsolutions, then $v=\inf_{i\in\cI} u_i$ and $V=\sup_iu_i $ are subsolutions.\medskip 
\item[3)]    If $u_i$  is  an {  equibounded}  family of supersolutions, then    $V=\inf_{i\in\cI} u_i $    is a supersolution.
\end{itemize} 
\end{prop}

     \begin{proof}
     For completeness, let us explain the last point. It is a direct consequence of the monotonicity of $\cT$: 
     $$\cT V = \cT \inf_iu_i \leq \inf_i \cT u_i \leq \inf_i u_i = V.$$
     \end{proof}

   The next two propositions are contained in \cite[Theorem 2.32]{Z}, and \cite[Theorem 2.29]{Z}, respectively.  
   \begin{prop}\label{subpeierls} 
    If $u$ is a subsolution then we have   
    $$u(x)-u(y) \le h(y,x). $$
 \end{prop}
      This is \cite[Theorem 2.29]{Z}.  
       \begin{prop}
 \label{solpeierls} 
 Let $y\in M$. 
\begin{itemize}
\item[1)] The function $ h(y, \cdot ) $ is a solution \eqref{appendix eqHJ}.   
\item[2)] The function $ -h(\cdot,y) $ is a  continuous subsolution  of \eqref{appendix eqHJ}.
\item[3)] The family of functions $\big( -h(\cdot,y)\big)_{y\in M} $ is equi-bounded. 
\end{itemize}
\end{prop}
 The following proposition is from \cite[Theorem 13]{BB}.  We provide here a different proof using subsolutions.
 \begin{prop}
\label{supportcontained}
The support of a projected Mather measure is contained in the projected Aubry set.

  \end{prop}
  
  \begin{proof}
  Let $(x,y) $ be in the support of a Mather measure $\tilde \mu $. Then we have 
  \begin{equation}
  \label{sub=c}
  u(y)-u(x)=c(x,y).
  \end{equation}
  for all { \it continuous subsolutions.} Indeed, since $u$ is a subsolution,  if we integrate with respect to $\tilde \mu$ we obtain 
  $$\int_{M\times M}\big( u(y)-u(x)\big)\, d\tilde \mu(x,y) \le \int_{M\times M} c(x,y)\, d\tilde \mu(x,y). $$
  But both sides are equal to zero, the left hand side because the measure is closed and the right hand side because it is a Mather measure.  So the inequality
  $u(y)-u(x) \le c(x,y) $ is an equality $\tilde \mu$-almost everywhere, and since they are continuous functions the equality is everywhere   on the support of $\tilde \mu$.
   So $u(y)=u(x)+c(x,y) \ge \cT( u)(y)$, but  $u$ being a subsolution we conclude that $\cT(u)(y) =u(y) $. Applying \eqref{sub=c} to $\cT (u)$ we obtain 
  \begin{equation}
  \label{sub=c2}
  \cT(u)(y)-\cT(u)(x)=c(x,y),
  \end{equation}
  yielding that also in $x$ we have $\cT(u)(x) =u(x)$.  
 
 We have thus proven that, for a given Mather measure $\tilde\mu$ and for every continuous subsolution $u$ of \eqref{appendix eqHJ}, 
 \[
 \cT(u)(x)=u(x)\quad\hbox{and}\quad \cT(u)(y)=u(y)\quad\hbox{for every $(x,y)\in\D{supp}(\tilde\mu)$.}
 \] 

We will prove that this implies that the points $x$ and $y$ are in the projected Aubry set $\cA$, that is $h(x,x)=h(y,y)=0$. 
Indeed, suppose that  $\cT(u)(z) =u(z) $ for all continuous subsolution $u$ of \eqref{appendix eqHJ}.
 By using   $u(\cdot )=-h(\cdot, z)$ and   by an immediate induction we infer that  for every positive integer $n$
 $$ u(z) = \cT^n(u)(z) = \inf_{z'\in M} u(z') + c_n(z',z) = u(z_n)+c_n(z_n,z), $$
 where $z_n$ is a suitable point that exists by compactness of $M$ and continuity of $u$ and $c_n$. Up to extracting a subsequence, $(n_k)_k$, we may assume that $z_{n_k}$ converges to some $\xi\in M$ and then get
 $$u(z) = \liminf_k u(z_{n_k}) + c_{n_k}(z_{n_k},z) \geq u(\xi)+h(\xi,z).$$
 We just proved that $-h(z,z) \geq -h(\xi,z)+h(\xi,z)=0$. The other inequality being always true, we get $h(z,z)=0$ and $z\in\cA$.
 \end{proof}

As a corollary of Proposition \ref{continuitysol} and  Proposition \ref{supportcontained} we obtain

\begin{cor}\label{integrability}
Any subsolution of \eqref{appendix eqHJ} is integrable with respect to any projected Mather measure.
\end{cor}
  
  We remark that the main point is that the measurability of the subsolution is not  a priori required.

  \subsection{Maximum Principle}

 \begin{prop}\label{Maxprinciple}
 If $v$ is  a subsolution and $w$  is a continuous   supersolution of \eqref{appendix eqHJ} such that $v\le w$ on  the projected Aubry set $\mathcal A$, then $v \le w$ on $M$. 
 \end{prop}
  
\begin{proof}
Let us set $u:=\cT(v)$. Then $u$ is still a subsolution of \eqref{appendix eqHJ}  such that $u=v$ on $\mathcal A$, see Proposition \ref{continuitysol}. 
Since $v\leq u$ on $M$, it suffices to prove the statement with $u$ in place of $v$. The advantage is that $u$ is always a continuous function, while the subsolution $v$ we started with may be discontinuous a priori.  

Let us proceed to prove the statement with $u$ in place of $v$. Let $x$ be an arbitrary point in $M$.  Since $w$ is a supersolution, we can find a point $x_{-1}$ such that
  $$ w(x_{-1}) + c(x_{-1}, x )\le  w(x).$$
Arguing inductively, we construct a sequence $ \bar x $  in $S_\infty(x)$ such that 
 $$ w(x_{-i-1}) + c(x_{-i-1},x_{-i})  \le w(x_{-i})\quad\hbox{for every $i\in\N$}  $$
 and  
 $$ w(x_{-n}) +  c (x_{-n}, x_{-n+1})+...+c(x_{-1},x)  \le w(x). $$
 On the other hand, since $u$ is  a subsolution, we have  
  \begin{eqnarray*}
   u(x) & \le & u (x_{-1}) + c(x_{-1},x) \\
   u(x_{-1})  & \le & u (x_{-2}) + c(x_{-2},x_{-1}) \\
    & \vdots &   \\
    u(x_{-n+1}) & \le &u(x_{-n}) + c(x_{-n}, x_{-n+1}) 
   \end{eqnarray*}
  for every $n\in \N$, so 
    $ u(x) \le u(x_{-n}) + c (x_{-n}, x_{-n+1})+...+c(x_{-1},x)$.
  Therefore  $$ w(x) -u(x) \ge w(x_{-n})- u(x_{-n}).$$
 We claim that any accumulation point of $\bar x$ belongs to the projected Aubry set.  This is enough to conclude. Indeed,  let $\left(x_{-n_k}\right)_k$ be an appropriate subsequence converging to a point $z\in\mathcal A$. Then  
$$ 
w(x) -u(x) \ge \lim_{k\to +\infty} w(x_{-n_k})- u(x_{-n_k})=  w(z)- u(z) \ge 0.
$$
Let us then prove the claim.
Let $z$ be an accumulation point of $\bar x$. Then we can find    two diverging sequences $(n_k)_k$ and $(M_k)_k$ such that the points $x_{-n_k-M_k}$ and $x_{-n_k}$ 
are converging to $z$ as $k\to +\infty$. We have 
     \begin{eqnarray*}
   &h_{M_k}&\!\!\!\!\! (x_{-n_k-M_k}, x_{-n_k}) \le  c(x_{-n_k-M_k}, x_{-n_k-M_k+1})+ ....+c(x_{-n_k+1},x_{-n_k}) \\
    & \le &  -w(x_{-n_k-M_k}) + w(x_{-n_k-M_k+1})+... -w(x_{-n_k+1})+ w(x_{-n_k})   \\
 &=&    -w(x_{-n_k-M_k}) + w(x_{-n_k}),
  \end{eqnarray*}
  so sending $k\to +\infty$ we infer that $h(z,z)=0$, i.e. $z$ belongs to the projected Aubry set $\mathcal A$, as it was claimed.  
\end{proof}
\end{appendix}

\bibliography{discount}
\bibliographystyle{siam}

 \end{document}